\documentclass[
final
]{dmtcs-episciences}

\usepackage[utf8]{inputenc}

\usepackage{amsmath,amsfonts,amssymb,amsthm}
\usepackage[capitalize]{cleveref}
\usepackage{hyperref}
\usepackage{thmtools}
\usepackage{xcolor}
\usepackage{subcaption}
\usepackage{float}
\usepackage{ytableau}
\usepackage{diagbox}

\graphicspath{ {./figures/} }

\theoremstyle{definition}

\newtheorem{thm}{Theorem}
\newtheorem{conj}[thm]{Conjecture}
\newtheorem{defn}[thm]{Definition}
\newtheorem{rem}[thm]{Remark}
\newtheorem{lem}[thm]{Lemma}
\newtheorem{prop}[thm]{Proposition}
\newtheorem{cor}[thm]{Corollary}
\newtheorem{example}[thm]{Example}

\numberwithin{thm}{section}
\numberwithin{defn}{section}
\numberwithin{rem}{section}
\numberwithin{conj}{section}
\numberwithin{lem}{section}
\numberwithin{cor}{section}
\numberwithin{example}{section}
\numberwithin{prop}{section}
\numberwithin{figure}{section}
\numberwithin{table}{section}

\usepackage[round]{natbib}

\author{Daniel Chen\affiliationmark{1}
  \and Sebastian Ohlig\affiliationmark{2}}
\title{Associated Permutations of Complete Non-Ambiguous Trees}

\affiliation{
  University of Cambridge, UK\\
  University of Oxford, UK}
\begin{document}
\publicationdata{vol. 25:2}{2024}{24}{10.46298/dmtcs.11169}{2023-04-08; 2023-04-08; 2023-10-03}{2023-11-15}

\maketitle
\begin{abstract}
   We explore new connections between complete non-ambiguous trees (CNATs) and permutations. We give a bijection between tree-like tableaux and a specific subset of CNATs. This map is used to establish and solve a recurrence relation for the number of tree-like tableaux of a fixed size without occupied corners, proving a conjecture by Laborde-Zubieta. We end by establishing a row/column swapping operation on CNATs and identify new areas for future research.
\end{abstract}

\section{Introduction}

In this paper, we study complete non-ambiguous trees (CNATs), a combinatorial object defined by \cite{abs10}. Recent research has found a range of intriguing mathematical cross-connections, such as the bijection in \cite{dgg19} to fully-tiered trees of weight \begin{math}0\end{math}.

We will link CNATs to another combinatorial object, tree-like tableaux. Defined by \cite{abn13}, they have been found to have applications in the PASEP model of statistical mechanics (\cite{csw07}). In this context \cite{zub15} first analysed occupied corners, that is vertices of the tableaux that have no cells below them or to their right.

We bridge these two areas of research by giving an explicit bijection (\cref{thm:tltbij}), with the help of which we prove a formula for the number of tree like tableaux without occupied corners, \begin{math}a(n)\end{math}, that was conjectured in Remark 3.3 of \cite{zub15}. We verify that \begin{math}a(n)\end{math} is indeed given by:
\begin{displaymath}
    a(n)=\sum_{k=\lceil \frac{n-1}{2}\rceil}^n{ \binom{k+1}{n-k}(-1)^{n-k}k!}
\end{displaymath}

A more complete outline of our work follows below.

In \cref{sec:defns}, we introduce the notion of an associated permutation, matrix and leaf matrix of a given CNAT. Subsequently, we study properties of CNAT leaf matrices and derive a necessary and sufficient condition for a collection of vertices to be the set of leaves of at least one CNAT (\cref{thm:irred}). We then consider when a collection of vertices is the set of leaves of \emph{exactly} one CNAT, and we prove the striking result that there are exactly \begin{math}2^{n-2}\end{math} such sets (\cref{cor:totalunique}) for an \begin{math}n\times n\end{math} CNAT.

In \cref{sec:udcnms}, we focus on upper-diagonal CNATs, which are CNATs whose leaf matrix is anti-diagonal. We prove a bijection between these CNATs and permutations (\cref{thm:permbij}) and also a constructive bijection to tree-like tableaux (\cref{thm:tltbij}). This allows a recurrence relation to be established (\cref{def:f}, \cref{thm:recformula1}) which we go on to solve, proving the conjecture in Laborde-Zubieta's paper.

We also prove that any CNAT can be reduced to upper-diagonal form through a series of row and column swaps which keep it a valid CNAT at each step (\cref{thm:upperred}). Finally we provide suggestions for further research (\cref{sec:bnk}).

\section{CNATs and CNMs}

\subsection{Definitions}\label{sec:defns}
    
Consider the bidimensional grid \begin{math}\mathbb{N}\times\mathbb{N}\end{math}. For convenience, given any \begin{math}v=(x,y) \in \mathbb{N}\times\mathbb{N}\end{math} we define \begin{math}X(v)=x,\; Y(v)=y\end{math}.  Then the definition of a CNAT is recalled from \cite{abs10} as follows:

\begin{defn}\label{def:CNAT}
    A \emph{non-ambiguous tree} is a subset \begin{math}T \subset \mathbb{N}\times\mathbb{N}\end{math} satisfying the following constraints:

\begin{enumerate}
    \item (Existence of a root) \begin{math}(1,1)\in T\end{math}.
    \item (Non-ambiguity) \begin{math}\forall p\neq (1,1)\in T \end{math}, exactly one of the following is true:
    \subitem \begin{math}\exists\end{math} \begin{math}r\in T\end{math} s.t. \begin{math}X(r)=X(p)\end{math} and \begin{math}Y(r)<Y(p)\end{math}
    \subitem \begin{math}\exists\end{math} \begin{math}r'\in T\end{math} s.t. \begin{math}Y(r')=Y(p)\end{math} and \begin{math}X(r')<X(p)\end{math}

    In other words, each vertex other than the root has exactly one possible precursor. This motivates the adjective "non-ambiguous".
    
    \item (Minimality) If there exists a \begin{math}p\in T\end{math} s.t. \begin{math}X(p)=x\end{math} (resp. \begin{math}Y(p)=y\end{math}), then for all \begin{math}x'<x\end{math} (resp.  \begin{math}y'<y\end{math}) there is a \begin{math}p'\in T\end{math} s.t. \begin{math}X(p')=x'\end{math} (resp. \begin{math}Y(p')=y'\end{math}). Simply put, \begin{math}T\end{math} contains no empty rows or columns.
\end{enumerate}

If \begin{math}T\end{math} also satisfies the condition that every vertex has zero or two children, then it is referred to as a \emph{complete non-ambiguous tree} (CNAT). For convenience, we will display CNATs with the root vertex in the upper left corner.
\end{defn}

\begin{figure}[H]
    \centering
    \includegraphics[width=40mm]{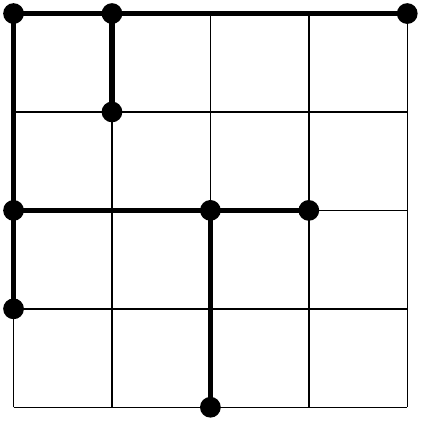}
    \caption{A complete non-ambiguous tree}
\end{figure}

By connecting each vertex to its precursor in the same row or column, every vertex except for the root has a unique parent, resulting in a binary tree structure rooted at \begin{math}(0,0)\end{math}. This implies that a CNAT with \begin{math}2n-1\end{math} vertices has \begin{math}n\end{math} leaves and \begin{math}n-1\end{math} internal vertices. Furthermore, every row and column has exactly one leaf, thus a CNAT with \begin{math}2n-1\end{math} vertices is contained in an \begin{math}n \times n\end{math} grid. We refer to the CNAT as having size \begin{math}n\end{math}, or dimensions \begin{math}n \times n\end{math}.

\begin{defn}[CNMs]
     For a complete non-ambiguous tree \begin{math}T\end{math}, a \emph{complete non-ambiguous matrix} (CNM) is the unique representation of \begin{math}T\end{math} as a binary matrix \begin{math}A = (a_{i,j})\end{math} defined by the relation:
    \begin{displaymath}
    a_{i,j} = 
    \begin{cases}
        1 & \text{if } (i,j) \in T\\
        0 & \text{otherwise}
    \end{cases}
    \end{displaymath}
    
We shall refer to a given \begin{math}a_{i,j} = 1\end{math} as a \emph{vertex} of the CNM and call it a \emph{leaf} (resp. \emph{internal vertex}) if the corresponding vertex \begin{math}(i,j)\end{math} is a leaf (resp. internal vertex) in \begin{math}T\end{math}.
\end{defn}

\begin{figure}[H]
\centering
    \begin{subfigure}[b]{0.5\textwidth}
        \centering
        \includegraphics{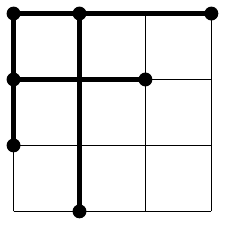}
        \caption{A \begin{math}4 \times 4\end{math} CNAT \begin{math}T\end{math}}
    \end{subfigure}%
    \begin{subfigure}[b]{0.5\textwidth}
        \centering
        \includegraphics{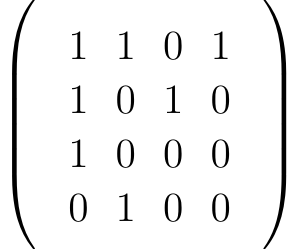}
        \caption{The CNM of \begin{math}T\end{math}}
    \end{subfigure}
    \caption{A CNAT and its associated matrix}
\end{figure}

CNATs and CNMs are in clear bijection with each other, but operating with CNMs allows to us to consider various matrix statistics and operations, leading to interesting results.

\begin{defn}[Leaf Matrices]
\label{def:pM}
For a given CNM \begin{math}M\end{math}, the \emph{leaf matrix} of \begin{math}M\end{math}, denoted by \begin{math}p(M)\end{math}, is the matrix obtained by setting all vertices but the leaves to zero in \begin{math}M\end{math}.
\end{defn}

Since there is exactly one leaf in every row and column, \begin{math}p(M)\end{math} always represents a permutation matrix. Interpreting this matrix as a permutation of the letters \begin{math}\{1,2,\dots n\}\end{math}, we obtain a permutation \begin{math}\pi(M)\in S_n\end{math}, which we refer to as the \emph{associated permutation}.
 
\begin{figure}[H]
\centering
    \begin{subfigure}{0.5\textwidth}
        \centering
        \includegraphics{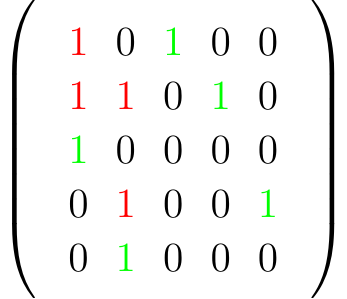}
    \end{subfigure}%
    \begin{subfigure}{0.5\textwidth}
        \centering
        \includegraphics{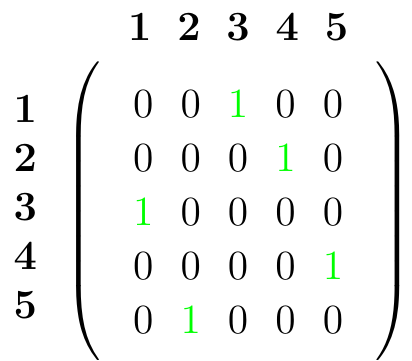}
    \end{subfigure}
    \caption{A CNM (left) and its leaf matrix (right)}
    \label{fig:cnmleaf}
\end{figure}

\begin{example}
Consider the matrix \begin{math}M\end{math} in \cref{fig:cnmleaf} above, with leaves coloured green and internal vertices coloured red. Its leaf matrix \begin{math}p(M)\end{math} discards the internal vertices of \begin{math}M\end{math}. By numbering the rows and columns from \begin{math}1\end{math} to \begin{math}5\end{math}, it can be seen that the associated permutation is \begin{math}\pi(M) = (13)(254)\end{math}, written in cycle notation.
\end{example}

\subsection{Determinants}
Given these preliminary definitions, we may formulate our first result.

\begin{prop}\label{prop:det}
The determinant of every complete non-ambiguous matrix \begin{math}M\end{math} is \begin{math}\text{sgn}(\pi(M))\end{math}, where \begin{math}\text{sgn}:S_n\to \{1,-1\}\end{math} is the sign of the permutation.
\end{prop}

\begin{proof}
   We consider the set of leaves in \begin{math}M\end{math} given by the matrix \begin{math}p(M)\end{math}. 
   By the non-ambiguity condition, for every leaf \begin{math}m_{i,j}\end{math} of \begin{math}M\end{math}, either the column of \begin{math}M_{\cdot,j}\end{math} or the row \begin{math}M_{i,\cdot}\end{math} is zero except for the vertex \begin{math}m_{i,j}\end{math}. Without loss of generality, assume it is the row.
   We add \begin{math}-M_{i,\cdot}\end{math} to all rows \begin{math}M_{k,\cdot}\end{math} in which \begin{math}m_{k,j}=1\end{math}, resulting in a matrix \begin{math}M'\end{math}. Since row operations do not affect the determinant, \begin{math}\det(M')=\det(M)\end{math}. Note that after these operations are finished, the matrix \begin{math}M'\end{math} has the same determinant as \begin{math}M\end{math}, but \begin{math}m_{i,j}\end{math} is the only non-zero vertex in \begin{math}M_{\cdot, j}\end{math} and \begin{math}M_{i,\cdot}\end{math}.
   
   After applying multiple such clearing operations, only the leaf matrix \begin{math}p(M)\end{math} remains in the structure. Its determinant, being the permutation matrix representing \begin{math}\pi\end{math} by definition, is exactly the value of \begin{math}\text{sgn}(\pi(M))\end{math}.
   
\end{proof}

\begin{figure}[H]
    \centering
    \begin{subfigure}{.4\textwidth}
        \centering
        \includegraphics{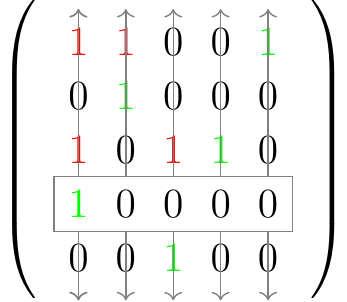}
        \subcaption{\begin{math}M\end{math}}
        \label{fig:pM}
    \end{subfigure}%
    \begin{subfigure}{.08\textwidth}
        \includegraphics{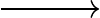}
    \end{subfigure}%
    \begin{subfigure}{.4\textwidth}
        \centering
        \includegraphics{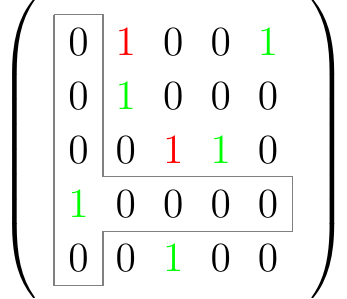}
        \subcaption{The matrix \begin{math}M'\end{math}}
        \label{fig:CNMleaf}
    \end{subfigure}
    
    \caption{The clearing process}
    \label{fig:test}
\end{figure}

Numerical exploration suggests the following result, which we have thus far been unable to prove.

\begin{conj}\label{conj:detparity}
    Let \begin{math}T(n)\end{math} be the number of \begin{math}n \times n\end{math} CNATS. If \begin{math}n>1\end{math} is odd, then there are \begin{math}\frac{T(n)}{2}\end{math} with determinant 1, and \begin{math}\frac{T(n)}{2}\end{math} with determinant -1.
\end{conj}

We computed the following data:
\begin{center}
\begin{tabular}{ c|c|c|c|c|c|c|c|c } 
\begin{math}n\end{math} & 1 & 2 & 3 & 4 & 5 & 6 & 7 & 8\\
\hline
\text{\# with det 1, A} & 1 & 0 & 2 & 17 & 228 & 4728 & 137400 & 5321889\\
\text{\# with det -1, B} & 0 & 1 & 2 & 16 & 228 & 4732 & 137400 & 5321856\\
\begin{math}T(n) = A+B\end{math} & 1 & 1 & 4 & 33 & 456 & 9460 & 274800 & 10643745\\
\begin{math}A-B\end{math} & 1 & -1 & 0 & 1 &  0 & -4 & 0 & 33\\
\end{tabular}
\end{center}

This data verifies the conjecture for \begin{math}n = 3,5,7\end{math}. It is known that \begin{math}T(n)\end{math} is related to the Bessel function \cite{abn13}, and is sequence A002190 in \cite{oeis}. Furthermore, when \begin{math}n \neq 1\end{math} is odd, \begin{math}T(n)\end{math} is even, because \begin{math}T(n)\end{math} (mod 2) is the characteristic function for powers of 2. This can be seen by the recurrence for $T(n)$ given in Proposition 9 of \cite{abs10}.

\begin{rem}
    The even terms of the sequence \begin{math}A-B\end{math} in the above table are \begin{math}-1, 1, -4, 33\end{math} which seem to be an alternating version of \begin{math}T(n)\end{math}.
\end{rem}

\section{Leaf Matrices}

In this section we will prove results about CNATs arising from specific leaf matrices, recalling  \cref{def:pM}.

\begin{defn}
A permutation matrix with associated permutation \begin{math}\sigma\end{math} on \begin{math}\{1,\dots, n\}\end{math} is \emph{irreducible} if there is no \begin{math}j\end{math} with \begin{math}1<j<n\end{math} such that \begin{math}\{1,\dots,j\}\end{math} is mapped to itself when applying \begin{math}\sigma\end{math}.
\end{defn}

\begin{rem}
    Sequence A003319 in \cite{oeis} counts the number of irreducible permutations of \begin{math}\{1\dots n\}\end{math}.
\end{rem} 

\begin{thm}[Irreducibility]
\label{thm:irred}
    A permutation matrix is the leaf matrix of some CNM if and only if it is irreducible.
\end{thm}

\begin{proof}
\begin{math}\left(  \Rightarrow \right)\end{math} We assume that there exists a CNM \begin{math}M\end{math} with \begin{math}p(M)\end{math} not irreducible and will arrive at a contradiction.

By definition, there exists a \begin{math}j<n\end{math} such that the set \begin{math}\{1\dots j\}\end{math} maps to itself; thus in \begin{math}p(M)\end{math}, the submatrix of the first \begin{math}j\end{math} rows and columns is a permutation matrix \begin{math}M'\end{math}. But then there can be no vertices below or to the right of \begin{math}M'\end{math} in \begin{math}M\end{math}, because the vertices in \begin{math}M'\end{math} represent leaves. This means that the tree represented by \begin{math}M\end{math} is disconnected, contradiction.

\begin{figure}[H]
    \centering
    \includegraphics{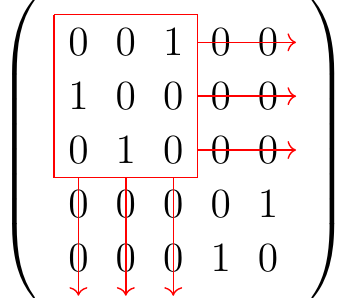}
    \caption{A permutation submatrix has no vertices below or to its right.}
    \label{fig:submatrix}
\end{figure}

\begin{math}\left( \Leftarrow \right)\end{math} We will prove by induction that every irreducible permutation matrix of size \begin{math}k\end{math} is the leaf matrix of at least one CNM of size \begin{math}k\end{math}. Note that the case \begin{math}k=2\end{math} is clear.
 
Fix \begin{math}k>2\end{math} and suppose that for \begin{math}2 \leq j < k\end{math} the statement holds. We will prove it holds true for \begin{math}k\end{math}. Fix a permutation matrix \begin{math}M\end{math} of size \begin{math}k\end{math} which is irreducible; we will add points to \begin{math}M\end{math} to construct a CNM whose leaf matrix is the original permutation matrix and be done. Initially, label all \begin{math}k^2\end{math} entries of \begin{math}M\end{math} as unmarked.

Consider the vertices \begin{math}l_x, l_y\end{math} in \begin{math}M\end{math} with maximal \begin{math}x,y\end{math} coordinate respectively. Since \begin{math}M\end{math} is irreducible, \begin{math}l_x\end{math} and \begin{math}l_y\end{math} are distinct.

Add the unmarked point \begin{math}p = (X(l_y), Y(l_x))\end{math} to \begin{math}M\end{math}, and label all entries in the column containing \begin{math}l_x\end{math} and row containing \begin{math}l_y\end{math} as marked. Note that at this point, the unmarked entries form a permutation submatrix. Then, for all leaves \begin{math}v_i \in M\end{math} with \begin{math}Y(v_i) > Y(l_x)\end{math} and \begin{math}X(v_i) > X(l_y)\end{math}, we mark all entries in the column and row containing \begin{math}v_i\end{math}, and freely choose to add either one of the (marked) points \begin{math}(X(v_i), Y(l_x))\end{math} or \begin{math}(X(l_y), Y(v_i))\end{math}. In other words, we choose to connect each \begin{math}v_i\end{math} either to the left or upwards. The unmarked entries still form a permutation submatrix, because we repeatedly marked rows and columns whose intersection is a vertex of \begin{math}M\end{math}.

Call this permutation submatrix of unmarked entries \begin{math}M'\end{math}, which has size strictly smaller than \begin{math}M\end{math}. Furthermore, \begin{math}M'\end{math} contains \begin{math}p\end{math}, and in \begin{math}M'\end{math} there are no unmarked vertices below and to the right of \begin{math}p\end{math}, because all \begin{math}v_i\end{math} were marked. We also claim that \begin{math}M'\end{math} is irreducible, the proof of which follows below.

Assume for sake of contradiction that \begin{math}M'\end{math} is not irreducible. Then by definition, it has some submatrix \begin{math}S\end{math} that is its own permutation matrix, and is not the entirety of \begin{math}M'\end{math}. Note that the point \begin{math}p\end{math} defined above is not vertically below or horizontally to the right of \begin{math}S\end{math} (see \cref{fig:submatrix}), so it must either be inside \begin{math}S\end{math}, or below \emph{and} to the right of (i.e. diagonally below) \begin{math}S\end{math}. If \begin{math}p\end{math} is in \begin{math}S\end{math}, then every position diagonally below \begin{math}S\end{math} is also diagonally below \begin{math}p\end{math}, so there are no unmarked vertices outside of \begin{math}S\end{math}, implying that \begin{math}S\end{math} is the entirety of \begin{math}M'\end{math}, contradiction. If \begin{math}p\end{math} is diagonally below \begin{math}S\end{math}, then \begin{math}S\end{math} was unchanged by all of the markings, so \begin{math}M\end{math} has the same submatrix \begin{math}S\end{math}, contradicting the fact that \begin{math}M\end{math} is irreducible.

Thus, by inductive hypothesis, \begin{math}M'\end{math} is the leaf matrix of some valid CNM. After adding the required points to \begin{math}M\end{math} that turn \begin{math}M'\end{math} into a CNM, we see that \begin{math}M\end{math} also becomes a valid CNM, because \begin{math}p\end{math} becomes the parent of the leaves \begin{math}l_x\end{math} and \begin{math}l_y\end{math} in \begin{math}M\end{math}, and each \begin{math}v_i\end{math} becomes a leaf. Thus the leaf matrix of \begin{math}M\end{math} is equal to the initial permutation matrix, as desired.

\end{proof}

\begin{rem}
    In the above proof, we used induction where the connection of the rightmost and bottom-most vertices reduced the problem to a smaller case. This is equivalent to an iterative argument, where at each iteration we perform the aforementioned connection of the rightmost and bottom-most vertices, and the connection of the \begin{math}v_i\end{math}. Thus we will refer to this as a "step" of the construction in the proof of sufficiency of \cref{thm:irred}.
\end{rem}

\begin{example}
\cref{fig:cnatconstruction} below illustrates one step of the construction process.

Any new vertices added are coloured red. We connect the rightmost and bottom-most vertices, then any vertices enclosed by this connection (denoted \begin{math}v_i\end{math} in the above proof, coloured blue) can connect to the left or upwards. In the given example, the connection occurs upwards. After removing rows and columns we obtain a smaller permutation matrix (highlighted in gray) which is irreducible, illustrating how the inductive argument yields the desired result.
\end{example}

\begin{figure}[H]
    \centering
    \includegraphics{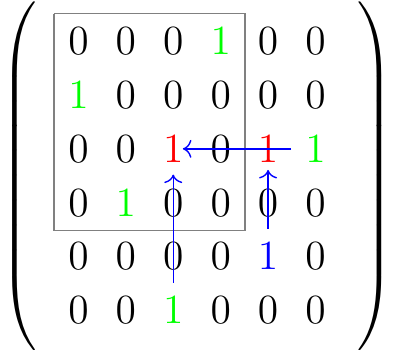}
    \caption{One step of the construction process}
    \label{fig:cnatconstruction}
\end{figure}

\begin{rem}\label{rem:twochoices}
    In the above proof of sufficiency, each \begin{math}v_i\end{math} has two possible connection choices (left or upwards). We may choose among these freely.
\end{rem}

\begin{defn}[L-subset]\label{def:Lset}
    For a matrix \begin{math}M\end{math} and \begin{math}k>1\end{math}, we define the \\ \begin{math}k^{th}\end{math}\emph{L-subset} to be:
    \begin{displaymath}
         \{m_{ij} \in M \mid \left(\left(i=k\right)\land \left(j\leq k\right)\right)\,\lor\, \left(\left(j=k\right)\land\, \left(i\leq k\right)\right)\}
    \end{displaymath}
    
\end{defn}
\begin{figure}[H]
    \centering
    \includegraphics{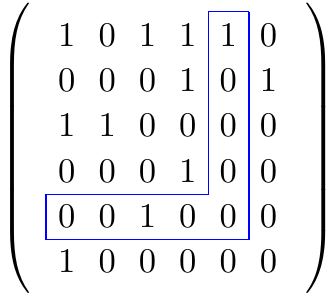}
    \caption{The fifth L-subset in a \begin{math}6\times 6\end{math} CNM}
    \label{fig:Lsetdemo}
\end{figure}

\begin{lem}\label{lem:nodiag}
    If a permutation matrix \begin{math}M\end{math} is the leaf matrix of exactly one CNAT, then it has no vertices on the diagonal.
\end{lem}

\begin{proof}
    Suppose \begin{math}M\end{math} represents a unique CNAT, and \begin{math}M\end{math} has a vertex \begin{math}v'\end{math} on the diagonal. We will show a contradiction.
    
    We use the constructive argument in the proof of sufficiency in \cref{thm:irred}.
    At each step, because \begin{math}M\end{math} is irreducible, \begin{math}v'\end{math} can never be the rightmost or bottom-most vertex, thus it must eventually be enclosed in the connection of those vertices. Then by \cref{rem:twochoices}, we can construct two possible CNATs by choosing to connect \begin{math}v'\end{math} to the left or upwards, contradicting uniqueness.
\end{proof}

\begin{rem}\label{rem:lastl}
    Note that \cref{lem:nodiag} implies that the \begin{math}n\end{math}-th L-subset of a leaf matrix of size \begin{math}n\end{math} always has two vertices, since both the last row and the last column must contain a vertex, which cannot be on the diagonal, meaning that the two vertices need to be distinct.
\end{rem}

\begin{lem}\label{lem:lsubone}
    If a permutation matrix \begin{math}M\end{math} of size \begin{math}n\end{math} is the leaf matrix of exactly one CNAT, then every L-subset other than the \begin{math}n\end{math}-th one has exactly one vertex.
\end{lem}

\begin{proof}
    We will induct on \begin{math}n\end{math}. Note that by manual verification, \begin{math}n=2\end{math} holds.
    
    Fix \begin{math}n>2\end{math}, and suppose the result holds for \begin{math}n-1\end{math}. We will show it holds for \begin{math}n\end{math}.
    
    Suppose the matrix \begin{math}M\end{math} represents a unique CNAT. We use the construction in the proof of sufficiency in \cref{thm:irred}. By uniqueness, there must not be any vertices enclosed in the box defined connecting the rightmost and bottom-most vertices. But this means that after this connection step, the size of \begin{math}M\end{math} decreases by exactly one, that is \begin{math}M'\end{math} has size \begin{math}n-1\end{math}. Let \begin{math}p\end{math} be the connection point of the rightmost and bottom-most vertices in \begin{math}M\end{math}, and let the \begin{math}n-1\end{math}-th L-subset be \begin{math}J\end{math}. Note that by \cref{rem:lastl}, \begin{math}J\end{math} has two vertices in \begin{math}M'\end{math}. \\
    We distinguish two cases:
    
    \noindent \emph{Case 1}: \begin{math}p\not\in J\end{math}
    
    We deduce that \begin{math}J\end{math} must have two vertices, \begin{math}u,v\end{math} in \begin{math}M\end{math}. Without loss of generality, let \begin{math}X(u) < X(v)\end{math}. Since \begin{math}u,v\end{math} were not enclosed by the connection of the rightmost and bottom-most vertices, we must have \begin{math}X(u) < X(v_y)\end{math}  and \begin{math}Y(v) < Y(v_x)\end{math}. But then, in the next step of the construction, we have \begin{math}X(u) < X(p)\end{math} and \begin{math}Y(v) < Y(p)\end{math} so \begin{math}p\end{math} will be enclosed and then we have a choice, contradiction. So Case 1 cannot occur.
    
    \noindent \emph{Case 2}: \begin{math}p\in J\end{math}
    
    By the inductive hypothesis, every L-subset of \begin{math}M'\end{math} other than \begin{math}J\end{math} has one vertex. But \begin{math}J\end{math} has two vertices in \begin{math}M'\end{math} by \cref{rem:lastl}, so it has one vertex in \begin{math}M\end{math}. Thus, every L-subset in \begin{math}M\end{math} other than the \begin{math}n\end{math}-th indeed contains exactly one vertex.
    
\end{proof}

\begin{thm}
\label{thm:totalzero}
    A permutation matrix \begin{math}M\end{math} of size \begin{math}n\end{math} is the leaf matrix of exactly one CNAT \begin{math}\iff\end{math} every L-subset other than the \begin{math}n\end{math}-th one has exactly one vertex, and there are no vertices on the diagonal.
\end{thm}

\begin{proof}
    \begin{math}\left( \Rightarrow \right)\end{math} Consequence of \cref{lem:nodiag} and \cref{lem:lsubone}.
    
    \begin{math}\left( \Leftarrow \right)\end{math} Suppose every L-subset of \begin{math}M\end{math} (other than the \begin{math}n\end{math}-th) contains exactly one vertex and there are none on the diagonal.

    Then, for every \begin{math}k\end{math}-th L-subset (\begin{math}1 < k \leq n\end{math}), we claim that the submatrix of the first \begin{math}k-1\end{math} rows and columns must have exactly one empty column and one empty row.
    
    We prove this claim by induction. Note that the base case, \begin{math}k=2\end{math}, is true since by assumption there are no vertices on the main diagonal. For the inductive step, fixing \begin{math}k>2\end{math}, the addition of an L-subset to the first \begin{math}k-1\end{math} rows and columns will result in a submatrix of size \begin{math}k\end{math}. We then place the vertex in the \begin{math}k\end{math}-th L-subset - there are two choices since we must place it below the empty column in the submatrix of size \begin{math}k-1\end{math}, or to the right of the empty row (anywhere else will be below or to the right of an existing leaf). This results in one empty row and one empty column in the submatrix of size \begin{math}k\end{math}.
    
    By the above, every leaf has exactly one possible parent, so \begin{math}M\end{math} represents at most one CNAT. But by \cref{thm:irred} \begin{math}M\end{math} represents \emph{at least} one CNAT, due to the fact that in every submatrix of size \begin{math}k\end{math}, the \begin{math}k\end{math}-th row or column is empty, since every L-subset only contains one leaf and there are none on the diagonal. Hence, \begin{math}M\end{math} represents a unique CNAT.
    
\end{proof}

\begin{cor}\label{cor:totalunique}
    There are \begin{math}2^{n-2}\end{math} permutation matrices of size \begin{math}n\end{math} which are the leaf matrix of exactly one CNAT.
\end{cor}

\begin{proof}
    For each \begin{math}2 \leq k < n\end{math}, by the above we have two choices of where to place the vertex in the \begin{math}k\end{math}-th L-subset. The root cannot be a leaf, and we have no choice of where to place the vertex in the \begin{math}n\end{math}-th L-subset since we must fill in the final empty row and column. Multiplying all options yields the desired result.
    
\end{proof}

\section{Upper-Diagonal CNMs}\label{sec:udcnms}

In this section, we will consider CNMs in which the leaf matrix is anti-diagonal. Bijections will be given from these to other combinatorial objects. To do this, we will need the definitions of a tree-like tableau and its occupied corners according to \cite{abn13} as well as those of fully-tiered trees and their weights, both as used in \cite{dgg19}.

\subsection{Definitions}

\begin{defn}[Upper-Diagonal CNMs]
An \emph{upper-diagonal CNM} is a CNM \begin{math}M\end{math} in which the leaf matrix \begin{math}p(M)\end{math} is anti-diagonal.
\end{defn}

\begin{figure}[H]
    \centering
    \includegraphics{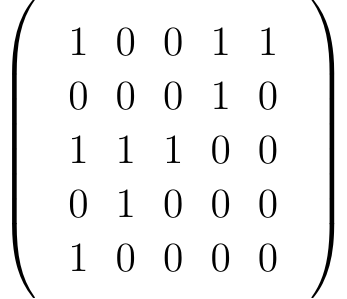}
    \caption{An upper-diagonal complete non-ambiguous matrix}
    \label{fig:UDCNMexamp}
\end{figure}

\begin{defn}[Second Diagonal]
The \emph{second diagonal} of an upper-diagonal CNM \begin{math}M = (m_{ij})\end{math} with size \begin{math}n\end{math} is the set of matrix entries \begin{math}\{m_{1,n-1}, m_{2,n-2}, \dots, m_{n-1,1}\}\end{math}.
\end{defn}

\begin{defn}[Tree-Like Tableaux and Occupied Corners]
    A \emph{tree-like tableau} of size \begin{math}n\end{math} is a collection of left-aligned rows of cells, where each cell is either empty or contains a point (referred to as an empty cell and pointed cell respectively) and there are \begin{math}n\end{math} pointed cells. The following conditions must also be satisfied:
    
    \begin{enumerate}
        \item (Existence of a root) The top-left cell contains a point, referred to as the root.
        \item (Non-ambiguity) Each pointed cell other than the root has another pointed cell in the same row and to the left, or another in the same column and above, but not both.
        \item (Minimality) There are no rows or columns which contain no pointed cells.
    \end{enumerate}
    
    A cell of the tree-like-tableau is a \emph{corner} if there are no cells below or to the right of it. If a corner cell contains a point, then it is an \emph{occupied corner}.
    
\end{defn}

\begin{figure}[H]
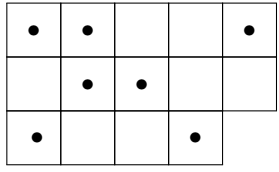

    \centering
    \ytableausetup{boxsize=2em}
    \begin{ytableau}
        \bullet & \bullet & & & \bullet\\
         & \bullet & \bullet & & \\\
        \bullet & & & \bullet
    \end{ytableau}
    \caption{A tree-like tableau of size 7}
\end{figure}

Note the similarities between the definition of tree-like tableaux and CNATs.

The following two definitions are taken from \cite{dgg19}. 

\begin{defn}[Tiered Trees]
    For \begin{math}n \in \mathbb{N}\end{math}, let \begin{math}[n] = \{1, \dots, n\}\end{math}.
    A \emph{tiered tree} of size \begin{math}n\end{math} is a labelled tree on \begin{math}n\end{math} vertices with vertex set \begin{math}V = [n]\end{math}, together with a \emph{tiering function} \begin{math}t : V \mapsto [n]\end{math} such that if \begin{math}u,v \in V\end{math} with \begin{math}u\end{math} adjacent to \begin{math}v\end{math} and \begin{math}u<v\end{math}, then \begin{math}t(u) < t(v)\end{math}. If \begin{math}t(u)=a\end{math}, then \begin{math}u\end{math} is said to be on tier \begin{math}a\end{math}.
    
    If \begin{math}t\end{math} is bijective, then we say the tree is \emph{fully-tiered} and call it a \emph{fully-tiered tree}.
\end{defn}

The tiers of a fully-tiered tree can be graphically represented by ordering vertices into vertical levels, with one vertex per level, where \begin{math}u\end{math} is on the j-th level if \begin{math}t(u)=j\end{math}.

\begin{figure}[H]
    \centering
    \includegraphics{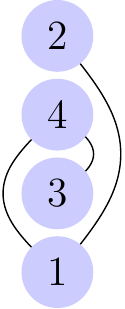}
    \caption{An example of a fully-tiered tree}
    \label{fig:FTTexamp0}
\end{figure}

\begin{defn}[Weights]
    \label{def:weight}
   We define the weight of a fully-tiered tree \begin{math}T\end{math} with tiering function \begin{math}t\end{math} as follows:  Let \begin{math}v\end{math} be the vertex of \begin{math}T\end{math} with minimal label and let \begin{math}T_1, \dots, T_k\end{math} be the connected components of \begin{math}T \setminus \{v\}\end{math}. The \emph{weight} of \begin{math}T\end{math}, \begin{math}\omega(T)\end{math}, is defined recursively as follows:
    
    \begin{displaymath}
    \omega(T) = 
    \begin{cases}
         0 & \text{if } \lvert T \rvert = 1\\
        \sum_{i=1}^k \left( \omega_i + \omega(T_i) \right) & \text{otherwise}
    \end{cases}
    \end{displaymath}
    
    The \begin{math}\omega_i\end{math} are defined like as follows: for each \begin{math}T_i\end{math} let \begin{math}u_i\end{math} be the unique vertex in \begin{math}T_i\end{math} to which \begin{math}v\end{math} is connected.  Then:
    \begin{displaymath}\omega_i=\vert\{u_j\in T_i\mid t(u_j)>t(v)\,\land\, u_j<u_i\}\vert\end{displaymath}
    In other words, \begin{math}\omega_i\end{math} is the zero-indexed position of \begin{math}u_i\end{math} in the set of all vertices in \begin{math}T_i\end{math} that could have been connected to \begin{math}v\end{math}.
    
\end{defn}

\begin{example} 
    We compute the weight of the fully-tiered tree in \cref{fig:FTTexamp0}. 
    
    The minimal vertex, 1, is connected to the components \begin{math}T_1, T_2\end{math} with vertex sets \begin{math}\{2\}\end{math} and \begin{math}\{3,4\}\end{math} respectively. \begin{math}\omega(T_1)=0\end{math} by definition, since \begin{math}\vert T_1\vert = 1\end{math}, and since there is only one vertex in \begin{math}T_1\end{math} that 1 could be connected to, we have \begin{math}\omega_1=0\end{math}. In \begin{math}T_2\end{math}, since 1 could have been connected to 3 instead of 4, we have \begin{math}\omega_2=1\end{math}. It can now be checked that \begin{math}\omega(T_2) = 0\end{math}, so the total is
    \begin{displaymath}\omega(T) = \left( \omega_1+\omega(T_1) \right) + \left( \omega_2 + \omega(T_2) \right) = (0+0)+(1+0)=1\end{displaymath}
\end{example}

\begin{example}
    As seen in \cref{fig:FTTexamp}, there are 5 fully-tied trees on three vertices, with weights 0, 1, 0, 0, 0 respectively:
    
    \begin{figure}[H]
    \centering
    \begin{subfigure}[b]{0.15\textwidth}
        \centering
        \includegraphics{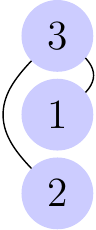}
    \end{subfigure}
    \begin{subfigure}[b]{0.15\textwidth}
        \centering
        \includegraphics{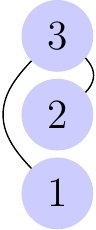}
    \end{subfigure}
    \begin{subfigure}[b]{0.15\textwidth}
        \centering
        \includegraphics{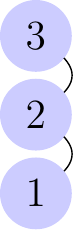}
    \end{subfigure}
    \begin{subfigure}[b]{0.15\textwidth}
        \centering
        \includegraphics{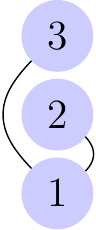}
    \end{subfigure}
    \begin{subfigure}[b]{0.15\textwidth}
        \centering
        \includegraphics{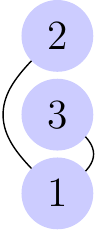}
    \end{subfigure}
    \caption{All fully-tiered trees on three vertices}
    \label{fig:FTTexamp}
\end{figure}

\end{example}

\begin{rem}
    For convenience, we will refer to vertices of a fully-tiered tree by their labels: let vertices \begin{math}v_1, v_2\end{math} have labels \begin{math}a,b\end{math} respectively with \begin{math}a>b\end{math}, then we say \begin{math}v_1\end{math} is larger than \begin{math}v_2\end{math}, and \begin{math}v_2\end{math} is smaller than \begin{math}v_1\end{math}.
\end{rem}

\begin{defn}[Upper-Diagonal FTTs]\label{def:udftt}
    We call a fully-tiered tree an \emph{upper-diagonal FTT} if it has weight 0 and for all its vertices \begin{math}v\end{math}, we have \begin{math}t(v) = v\end{math}.
\end{defn}

\begin{figure}[H]
    \centering
    \includegraphics{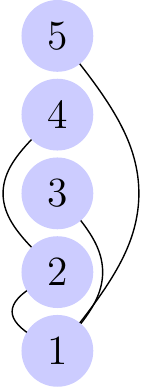}
    \caption{An upper-diagonal FTT}
    \label{fig:UDFTTexamp}
\end{figure}

\subsection{Bijections to Other Objects}
We give bijections between upper-diagonal CNMs, permutations and tree-like tableaux.

\begin{lem}
\label{lem:udftt}
    If \begin{math}n\end{math} is the vertex of the highest level in an upper-diagonal FTT of size \begin{math}n\end{math}, then \begin{math}n\end{math} is a leaf.
\end{lem}

\begin{proof}
    Suppose there is some \begin{math}n\end{math} such that there exists an upper-diagonal FTT, say \begin{math}T\end{math}, on \begin{math}n\end{math} vertices in which the vertex with maximal label \begin{math}n\end{math} is not a leaf. Consider the smallest such \begin{math}n\end{math}; we will show a contradiction by finding a smaller upper-diagonal FTT on fewer vertices in which the highest label is not a leaf.

    Note that since the vertex \begin{math}n\end{math} has degree at least 2 by assumption, \begin{math}n \geq 3\end{math}. Manually checking all cases for \begin{math}n=3\end{math}, we see that the vertex 3 is indeed always a leaf.

    For the rest of the proof, let \begin{math}T'\end{math} be the graph obtained from \begin{math}T\end{math} by deleting the vertex \begin{math}1\end{math} and all edges connected to it.

    \emph{Case 1}: \begin{math}1\end{math} is connected to \begin{math}n\end{math}.
    
    Since \begin{math}n\end{math} is not a leaf, there exists a neighbour \begin{math}u\end{math} of \begin{math}n\end{math} in T, which is distinct from 1.
    Let \begin{math}T_1\end{math} be the connected component of \begin{math}T'\end{math} containing \begin{math}u\end{math} and \begin{math}n\end{math}. Examining weights and the definition of the \begin{math}\omega_i\end{math}, we have \begin{math}\omega_1>0\end{math} since \begin{math}1\end{math} could have been connected to a smaller vertex \begin{math}T_1\end{math}, \begin{math}u\end{math}. By definition of weight, \begin{math}\omega(T) = \sum_{i=1}^k \left( \omega_i + \omega(T_i) \right)\geq \omega_1\geq 1\end{math}. But \begin{math}T\end{math} is an upper-diagonal FTT and thus has \begin{math}\omega(T) = 0\end{math} by definition, contradiction.

    \emph{Case 2}: \begin{math}1\end{math} is not connected to \begin{math}n\end{math}.
    
    Without loss of generality, let \begin{math}T_1\end{math} be the connected component of \begin{math}T'\end{math} containing \begin{math}n\end{math}. Note that \begin{math}n\end{math} has degree \begin{math}\geq 2\end{math} in \begin{math}T_1\end{math} and is the largest vertex of \begin{math}T_1\end{math}. Using \cref{def:weight} and the fact that \begin{math}T\end{math} has weight zero:
    \begin{displaymath}  \omega(T_0) = \sum_i (\omega_i + \omega(T_i))=0  \end{displaymath}
    All these sum terms are non-negative, so \begin{math}\omega(T_1)=0\end{math}.
    After relabelling the vertices of \begin{math}T_1\end{math} such that every vertex \begin{math}v\end{math} is in tier \begin{math}v\end{math}, since \begin{math}\omega(T_1) = 0\end{math} we have that \begin{math}T_1\end{math} (which is non-trivial since it contains at least two vertices) is an upper-diagonal FTT on fewer vertices than \begin{math}T\end{math} with the maximal vertex not being a leaf, contradicting the minimality of \begin{math}T\end{math}.
\end{proof}

\begin{lem}\label{lem:leafmeansudftt}
If \begin{math}T\end{math} is an upper-diagonal FTT of size \begin{math}n\end{math}, then adding the vertex \begin{math}n+1\end{math} in tier \begin{math}n+1\end{math} to \begin{math}T\end{math} gives another upper-diagonal FTT, \begin{math}T'\end{math}.
\end{lem}

\begin{proof}
Note that \begin{math}T \cup \{n+1\}\end{math} (i.e. \begin{math}T'\end{math}) is a fully-tiered tree with \begin{math}t(v) = v\end{math} for all its vertices \begin{math}v\end{math}. We will show that \begin{math}T'\end{math} has weight \begin{math}0\end{math} by strong induction. This is true for \begin{math}n=1\end{math}, \begin{math}n=2\end{math} by inspection, so fix \begin{math}n>2\end{math} and suppose the result holds for up to \begin{math}n-1\end{math}; we will prove it for \begin{math}n\end{math}.

\begin{math}T'\end{math} is a fully-tiered tree with \begin{math}t(v) = v\end{math} for all vertices \begin{math}v\end{math} (because \begin{math}T\end{math} is). Consider the effect of adding the vertex \begin{math}n+1\end{math} to \begin{math}T\end{math} on the computation of the weight.

If \begin{math}n+1\end{math} is connected to 1, \begin{math}n+1\end{math} becomes its own connected component after the removal of \begin{math}1\end{math} during the computation process, and so \begin{math}T\end{math} has weight 0 since the other connected components remain the same.

If \begin{math}n+1\end{math} is not connected to 1, let \begin{math}C\end{math} be the connected component of \begin{math}T' \setminus \{1\}\end{math} which contains the vertex connected to \begin{math}n+1\end{math} in \begin{math}T\end{math}. All other components are unaffected by the addition of \begin{math}n+1\end{math}, and so contribute 0 to the weight sum. But flattening the labels of \begin{math}C\end{math} and applying the induction hypothesis yields that \begin{math}C\end{math} has weight 0, therefore \begin{math}T'\end{math} has weight 0. So \begin{math}T'\end{math} is an upper-diagonal FTT since it satisfies all criteria in \cref{def:udftt}.

\end{proof}

\begin{rem}\label{rem:workstoo}
    By analogous argumentation, removing the maximal vertex from an upper-diagonal FTT of size \begin{math}n\end{math} results in an upper-diagonal FTT of size \begin{math}n-1\end{math}.
\end{rem}

\begin{thm}
    \label{thm:permbij}
    There is a bijection between upper-diagonal CNMs of size \begin{math}n\end{math} and permutations in \begin{math}S_{n-1}\end{math}.
\end{thm}

\begin{proof}
    \cite{dgg19} gives a bijection between CNATs and weight \begin{math}0\end{math} fully-tiered trees. In this, leaves of the CNAT at position \begin{math}(x,y)\end{math} are mapped to vertices with label \begin{math}x\end{math} and tier \begin{math}n+1-y\end{math}. Since upper-diagonal CNMs have \begin{math}X(v) = n+1-Y(v)\end{math} for all leaves \begin{math}v\end{math}, under this map the upper-diagonal CNMs are sent precisely to those fully-tiered trees with weight \begin{math}0\end{math} that have \begin{math}t(v) = v\end{math} for all vertices \begin{math}v\in\{1,2,\dots, n\}\end{math}. By \cref{def:udftt}, these are exactly the upper-diagonal FTTs.
    
    If we can find a map bijectively sending upper-diagonal FTTs of size \begin{math}n\end{math} to permutations in \begin{math}S_{n-1}\end{math}, then composing it with the above bijection results in a bijection between upper-diagonal CNMs of size \begin{math}n\end{math} and permutations in \begin{math}S_{n-1}\end{math}.

    Let \begin{math}T_n\end{math} be an upper-diagonal FTT of size \begin{math}n\end{math}. In the following, we will construct a sequence of upper-diagonal FTTs \begin{math}T_n, T_{n-1}, \dots, T_1\end{math}, represented by a sequence of \begin{math}n-1\end{math} distinct integers \begin{math}b_i\end{math}, with \begin{math}1\leq i\leq n-1\end{math}. This representation will be bijective, and interpreting the sequence as an element of \begin{math}S_{n-1}\end{math} gives rise to a permutation.
    
    For the sake of iteration, let \begin{math}S=\emptyset\end{math} and for each \begin{math}n > i \geq 1\end{math}, let \begin{math}T_i\end{math} be the tree obtained by removing the vertex with maximal label and its edges in \begin{math}T_{i+1}\end{math}; this vertex is exactly \begin{math}i+1\end{math}. Let \begin{math}a_i\end{math} be the label of the unique vertex to which \begin{math}i\end{math} connects in \begin{math}T_i\end{math}. We add \begin{math}a_{i+1}\end{math} to our sequence \begin{math}S\end{math}.
    
    Note that by \cref{lem:udftt}, \begin{math}i+1\end{math} was a leaf, so we removed exactly one vertex and one edge. Since \begin{math}T_{i+1}\end{math} comes from removing the vertex \begin{math}i+1\end{math} and its associated edge from \begin{math}T_i\end{math}, by \cref{rem:workstoo} it is also an upper-diagonal FTT. 
    
    Thus we may repeat this process until we reach the tree on one vertex. We obtain a sequence of \begin{math}n-1\end{math} numbers \begin{math}S=(a_1,a_2\dots a_n)\end{math}, where due to the maximality of the vertex \begin{math}i\end{math} in \begin{math}T_i\end{math}, we have \begin{math}1 \leq a_i \leq n-i\end{math}.
    
    Note that any such sequence satisfying \begin{math}1 \leq a_i \leq n-i\end{math} gives rise to a valid upper-diagonal FTT, given that \begin{math}T_i\end{math} is constructed from \begin{math}T_{i-1}\end{math} where the maximal vertex is a leaf, and thus \begin{math}T_i\end{math} has weight 0 by \cref{lem:leafmeansudftt}. We deduce that the set of all possible \begin{math}T_n\end{math} is in bijection with the set of sequences \begin{math}a_i\end{math} which satisfy \begin{math}1 \leq a_i \leq n-i\end{math}.
    
    Finally, to reversibly extract the permutation out of the \begin{math}a_i\end{math}, we define the sequence \begin{math}b_i\end{math} as follows. Let \begin{math}B_0=(1,\cdots, (n-1))\end{math}.
    
    For each \begin{math}1 \leq i < n\end{math}, we choose \begin{math}b_i\end{math} to be the \begin{math}a_i\end{math}th element of \begin{math}B_{i-1}\end{math} and define \begin{math}B_i\end{math} as the sequence \begin{math}B_{i-1}\end{math} in which the \begin{math}a_i\end{math}-th element was removed. Iterating this procedure for increasing values of \begin{math}i\end{math}, the sequence of \begin{math}B_i\end{math} is eventually exhausted. Since \begin{math}a_i\leq n-i=\vert B_{i-1}\vert\end{math}, this procedure always works. Note that at the end of this process, we have a sequence \begin{math}b_1,b_2\dots b_{n-1}\end{math}, where the \begin{math}b_i\end{math} are distinct values in \begin{math}\{1,2,\dots n-1\}\end{math}. Interpreting this as an element of \begin{math}S_{n-1}\end{math} yields a mapping from the \begin{math}a_i\end{math}'s to \begin{math}S_n\end{math}, which is seen to be bijective.
    
    We have shown that upper-diagonal CNMs of size \begin{math}n\end{math} are in bijection with upper-diagonal fully-tiered trees of size \begin{math}n\end{math}, which are in bijection with the \begin{math}a_i\end{math} as defined above, which are in bijection with permutations in \begin{math}S_{n-1}\end{math}.
    
\end{proof}

\begin{cor}\label{cor:totalUDCNM}
   There are \begin{math}(n-1)!\end{math} upper-diagonal CNMs of size \begin{math}n\end{math}.\\
\end{cor}

\begin{thm}\label{thm:tltbij}
    There is a bijection between upper-diagonal CNMs of size \begin{math}n\end{math}, and tree-like tableaux of size \begin{math}n-1\end{math}.
\end{thm}

\begin{proof}
    
    Firstly, given an upper-diagonal CNM of size \begin{math}n\end{math}, say \begin{math}M = (m_{i,j})\end{math}, we will injectively construct a tree-like tableau of size \begin{math}n-1\end{math}. Let us define \begin{math}A \end{math} to be the set of all elements of \begin{math}M\end{math} that are above the upper-diagonal:
    \begin{displaymath}
        A := \{ m_{i,j} \in M \mid i+j < n+1\}
    \end{displaymath}
    Then, we delete any empty rows and columns in \begin{math}A\end{math} to obtain the set \begin{math}S\end{math}:
    \begin{displaymath}
        S := \{a_{i,j} \in A \mid \exists\,a_{u,v} = 1 \in A \text{ s.t. } i=u \text{ or } j=v \}
    \end{displaymath}
    For the construction of the tree-like tableau, we create a collection of cells whose coordinates correspond to those of the points in \begin{math}S\end{math} (after relabelling to ensure every row is left-aligned, and every column is top-aligned), where the cell is filled if the corresponding point is a 1 and empty otherwise.
    
    Since we have taken a CNM and deleted the leaves, existence of a root, non-ambiguity and minimality are still all satisfied. (Note that completeness is no longer guaranteed). Thus the result is a valid tree-like tableau. Also note that this construction is injective and the resulting tree-like-tableau has size \begin{math}n-1\end{math}, because there are \begin{math}n-1\end{math} internal vertices of \begin{math}T\end{math}. But it is known that the number of tree-like tableaux of size \begin{math}n-1\end{math} is equal to \begin{math}(n-1)!\end{math} \cite{abn13}, which equals the number of upper-diagonal CNMs of size \begin{math}n\end{math} by \cref{thm:permbij}. Thus, our construction must be bijective.
\end{proof}

\begin{example}
    In \cref{fig:tltcons} below, we construct a tree-like tableau out of an upper-diagonal CNM (which has leaves in red and internal vertices in green). Everything on or below the upper-diagonal is removed to obtain the set \begin{math}A\end{math}, then empty rows and columns (the red lines) are removed and it is turned into a tree-like tableau.
\end{example}

\begin{figure}[H]
    \centering
    \begin{subfigure}[b]{0.5\textwidth}
        \centering
        \includegraphics{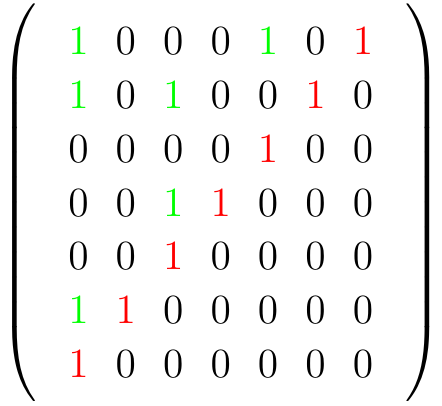}
        \caption{An upper-diagonal CNM}
    \end{subfigure}%
    \begin{subfigure}[b]{0.5\textwidth}
        \centering
        \includegraphics{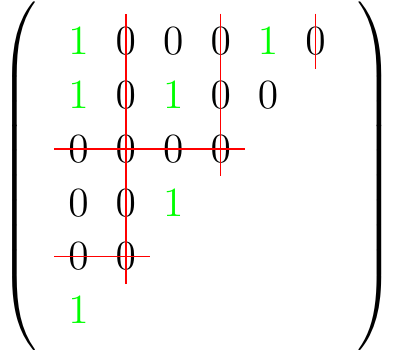}
        \caption{The set of points \begin{math}A\end{math}}
    \end{subfigure}
    
    \vspace{10mm}
    \begin{subfigure}{0.5\textwidth}
        \centering
        \ytableausetup{boxsize=2em}
        \begin{ytableau}
            \bullet & & \bullet \\
            \bullet & \bullet & \\
            & \bullet\\
            \bullet
        \end{ytableau}
        \caption{The constructed tree-like tableau}
    \end{subfigure}
    \caption{The construction process}
    \label{fig:tltcons}
\end{figure}

\begin{rem}
    The above construction sends vertices in the second diagonal of the upper-diagonal CNM to occupied corners in the tree-like tableau, and vice versa.
\end{rem}

\subsection{Proving a Conjecture by Laborde-Zubieta}

It was conjectured in \cite{zub15} in Remark 3.3 that the number of tree-like tableaux of size \begin{math}n\end{math} having no occupied corners is equal to the sequence A184185 in \cite{oeis}, which counts the number of permutations in \begin{math}S_n\end{math} having no cycles of the form \begin{math}(i, i+1, \dots, i+j)\end{math} for \begin{math}j \geq 0\end{math}.

Recall that in the bijective construction for \cref{thm:permbij}, occupied corners of a tree-like tableau were sent exactly to the vertices in the second diagonal of the corresponding upper-diagonal CNM and vice versa. Hence, the number of TLTs of size \begin{math}n\end{math} having no occupied corners is equal to the number of upper-diagonal CNMs of size \begin{math}n+1\end{math} with a completely zero second diagonal. Now, we establish and explicitly solve a recurrence relation for the latter value.

\begin{defn}
    \label{def:f}
    For \begin{math}0 \leq k < n-1\end{math}, we define \begin{math}f(n,k)\end{math} to be the number of \begin{math}n \times n\end{math} upper-diagonal CNMs \begin{math}M\end{math} with the first \begin{math}k\end{math} elements of the second diagonal starting with exactly \begin{math}k\end{math} zeros, followed by a one. For \begin{math}k = n-1\end{math}, we extend this definition by defining \begin{math}f(n,n-1)\end{math} to be the number of upper-diagonal CNMs with the second diagonal being completely zero.
\end{defn}

\begin{defn}\label{def:s}
    We define the sum function \begin{math}s(n,k)\end{math} by:
    \begin{displaymath} s(n,k)=\sum_{j=0}^{k}{f(n,j)}\end{displaymath}
\end{defn}

\begin{rem}\label{rem:totalsum}
    Since \begin{math}f(n,k)\end{math} counts the number of \begin{math}n \times n\end{math} upper-diagonal CNMs with \begin{math}k\end{math} leading zeroes in the second diagonal of the matrix, summing over all possibilities we deduce by \cref{cor:totalUDCNM} that \begin{math}s(n,n-1)=(n-1)!\end{math}
\end{rem}

\begin{thm}[Recursion Formula for \begin{math}f(n,k)\end{math}]\label{thm:recformula1}
\end{thm}
\begin{displaymath}
    f(n,k) =
    \begin{cases}
        0 & \text{if } k=-1 \text{ (by convention)}
        \\[10pt]
        \sum_{j=k-1}^{n-2}f(n-1,j) & \text{if } 0 \leq k < n-1
        \\[10pt]
        \sum_{j=0}^{n-4}(n-3-j)f(n-1,j) & \text{if } k=n-1
    \end{cases}
\end{displaymath}

\begin{proof}
    Consider an \begin{math}n \times n\end{math} upper-diagonal CNM with second diagonal containing \begin{math}k\end{math} leading zeros.
    
    \noindent \textit{Case 1}: \begin{math}k<n-1\end{math}

    By assumption, there exists a vertex \begin{math}v\end{math} of the matrix at position\\ \begin{math}(n-1-k,k+1)\end{math}. 
    Thus, applying the non-ambiguity condition to the leaves \begin{math}(n-k, k+1)\end{math} and \begin{math}(n-k-1, k+2)\end{math}, this forces the column \begin{math}k+2\end{math} and the row \begin{math}n-k\end{math} to be empty except for these anti-diagonal nodes.  Removing this column and row, we are left with a size \begin{math}n-1\end{math} upper-diagonal CNM, in which \begin{math}v\end{math} has become a leaf in the anti-diagonal of the new matrix, and in which the second diagonal starts with \emph{at least} \begin{math}k-1\end{math} zeroes, since all columns to the right of \begin{math}v\end{math} are shifted down by one, and the deletion of row \begin{math}k+2\end{math} removes a leading zero whenever \begin{math}k>0\end{math}. The cases for the number of leading zeroes are mutually exclusive, so summing over all values we obtain:
    \begin{displaymath}
        f(n, k) = \sum_{j=k-1}^{n-2}f(n-1,j)
    \end{displaymath}
    Note that the convention \begin{math}f(n,-1) = 0\end{math} allows us to incorporate \begin{math}k=0\end{math} in this recursive definition also.
    
    \noindent \textit{Case 2}: \begin{math}k=n-1\end{math}
    
    We evaluate \begin{math}f(n,n-1)\end{math} by subtracting all other cases from the total of \begin{math}(n-1)!\end{math} (by \cref{rem:totalsum}).
    \begin{align*}
        f(n, n-1) &= (n-1)! - \sum_{i=0}^{n-2}f(n,i)\\
        &= \left[(n-1)\sum_{i=0}^{n-2}f(n-1,i)\right] - \left[ \sum_{i=0}^{n-2} \sum_{j=i-1}^{n-2} f(n-1,j) \right] \;\text{ (by Case 1)}\\
        &= \sum_{j=0}^{n-4} (n-3-j) f(n-1,j) \;\text{ (by swapping double summation)}
    \end{align*}
\end{proof}

\begin{lem}\label{lem:reccase1}
    \begin{displaymath}f(n,0)=(n-2)!\end{displaymath}
\end{lem}

\begin{proof}
By Case 1:
\begin{align*}
    f(n,0) &= \sum_{j=-1}^{n-2}{f(n-1,j)}\\
    &= \underbrace{f(n-1,-1)}_\text{= 0} + \sum_{j=0}^{n-2}{f(n-1,j)}\\
\end{align*}
By \cref{rem:totalsum}, this evaluates to \begin{math}(n-2)!\end{math}, as required.
\end{proof}

\begin{lem}\label{lem:reccase2}
    For \begin{math}0 < k < n-1\end{math},
    \begin{displaymath}f(n,k)=f(n,k-1)-f(n-1,k-2)\end{displaymath}
\end{lem}

\begin{proof}
    \begin{align*}
        f(n-1,k-2) &= \sum_{j=k-2}^{n-2}{f(n-1,j)} - \sum_{j=k-1}^{n-2}{f(n-1,j)}\\
        &= f(n,k-1)-f(n,k) & \text{(by Case 1)}
    \end{align*}
    Rearranging yields the desired result.
\end{proof}

\begin{lem}\label{lem:reccase3}
\begin{displaymath}f(n,n-1)=f(n+1,n-1)-f(n,n-2)\end{displaymath}
\end{lem}

\begin{proof}
    Recalling that \begin{math}s(n-1,n-2) = (n-2)!\end{math} yields:
    \begin{equation}\label{eq:rec1}
    \begin{split}
        s(n-1,n-3) &= (n-2)!-f(n-1,n-2)
    \end{split}
    \end{equation}
    
    But if we consider \begin{math}\Delta = s(n,n-2) - s(n-1,n-3)\end{math}:
    \begin{align*}
        \Delta &= f(n,0) + f(n,n-2) - f(n-1,n-3)\\
        &\;\;\;\;\;\;+ \sum_{j=2}^{n-2}{\underbrace{\left(f(n,j-1)-f(n-1,j-2)\right)}_\text{= \begin{math}f(n,j)\end{math}}}\\
        &= s(n,n-2)-f(n,1)+f(n,n-2)-f(n-1,n-3)
    \end{align*}
    
    Cancelling \begin{math}s(n,n-2)\end{math} from both sides and rearranging yields:
    \begin{displaymath}
        s(n-1,n-3) = f(n,1)+f(n-1,n-3)-f(n,n-2)
    \end{displaymath}
    
    Substituting this into \cref{eq:rec1}:
    \begin{displaymath}
        f(n,1)+f(n-1,n-3)-f(n,n-2) = (n-2)! - f(n-1,n-2)
    \end{displaymath}
    
    But note that, by \cref{lem:reccase2}, we have \begin{math}f(n,1)=f(n,0)-f(n-1,-1)=(n-2)!\end{math}. Thus we can cancel the \begin{math}(n-2)!\end{math} term and rearrange:
    \begin{equation*}
        \begin{split}
            f(n-1,n-2) = f(n,n-2) - f(n-1,n-3)
        \end{split}
    \end{equation*}
    Shifting indices proves the lemma.
\end{proof}

\begin{rem}
    Combining \cref{lem:reccase1}, \cref{lem:reccase2} and \cref{lem:reccase3} allows us to greatly simplify our recurrence relation for \begin{math}f(n,k)\end{math}:
    \begin{equation*}
        f(n,k)= 
        \begin{cases}
        0 & \text{if}\; k<0\\
        (n-2)! & \text{if}\; k=0\\
        f(n,k-1)-f(n-1,k-2) & \text{if}\; 1 <k<n-1\\
        f(n+1,n-1)-f(n,n-2) & \text{if}\; k=n-1\\
        \end{cases}
    \end{equation*}
\end{rem} 

\begin{thm}\label{thm:binomrep}
For \begin{math}0\leq k<n-1\end{math}:
\begin{equation*}
    f(n,k)=\sum_{j=0}^k \binom{k-j}{j}(-1)^j(n-2-j)!
\end{equation*}
\end{thm}

\begin{proof}
For convenience, let \begin{math}g(n,k)\end{math} be the RHS of the above equation. We will induct on the lexicographic ordering of pairs \begin{math}(n,k)\end{math}.

\textit{Base case}: \begin{math}(n,k) = (2,0)\end{math} is true. Additionally, the statement is clearly true whenever \begin{math}k=-1\end{math}.

\textit{Inductive step:} Fix \begin{math}n, k\end{math} with  \begin{math}k \leq n-2\end{math} and suppose for all pairs \begin{math}(2,0) \leq (a,b) < (n,k)\end{math} we have \begin{math}f(a,b) = g(a,b)\end{math}. We will show \begin{math}f(n,k) = g(n,k)\end{math}.

If \begin{math}k=0\end{math}, then \begin{math}g(n,k) = (n-2)! = f(n,k)\end{math}.

Otherwise:
\begin{align*}
    f(n,k) &= f(n, k-1) - f(n-1, k-2) & \text{By \cref{lem:reccase2}}\\
    &= g(n, k-1) - g(n-1, k-2) & \text{By inductive hypothesis}\\
    &= \sum_{j=0}^{k-1} \binom{k-1-j}{j}(-1)^j(n-2-j)!\\
    &\;\;\;\;\;\; - \sum_{j=0}^{k-2} \binom{k-2-j}{j}(-1)^j(n-3-j)! & \text{By definition}
\end{align*}

Extracting the first term of the former sum and re-indexing yields
\begin{align*}
    f(n,k) &= (n-2)! + \sum_{j=0}^{k-2}{\binom{k-2-j}{j+1}(-1)^{j+1}(n-3-j)!}\\
    &\;\;\;\;\;\; - \sum_{j=0}^{k-2}{\binom{k-2-j}{j}(-1)^j(n-3-j)!}\\
    &= (n-2)! - \sum_{j=0}^{k-2} \left[ \binom{k-2-j}{j+1} + \binom{k-2-j}{j} \right] (-1)^j(n-3-j)!
\end{align*}

By Pascal's Identity, this implies
\begin{align*}
    f(n,k) &= (n-2)! - \sum_{j=0}^{k-2} \binom{k-1-j}{j+1} (-1)^j(n-3-j)!\\
    &= \sum_{j=0}^{k-1} \binom{k-j}{j} (-1)^j (n-2-j)!\\
    &= g(n,k)
\end{align*}
\end{proof} 

\begin{thm}\label{thm:f(nn-1)}
We have:
\begin{displaymath}f(n,n-1)=\sum_{j=0}^{n-1}{\binom{n-j}{j}(-1)^j(n-1-j)!}\end{displaymath}
\end{thm}
\begin{proof}
Recall \cref{lem:reccase3}:
\begin{displaymath}f(n,n-1)=f(n+1,n-1)-f(n,n-2)\end{displaymath}
We replace the RHS with the expressions given by \cref{thm:binomrep}:
\begin{align*}
        f(n,n-1) &= \sum_{j=0}^{n-1}{\binom{n-1-j}{j}(-1)^j(n-1-j)!}\\
        &\;\;\;\;\;\; - \sum_{j=0}^{n-2}{\binom{n-2-j}{j}(-1)^j(n-2-j)!}
\end{align*}
Reindexing the second sum and applying Pascal's Identity, we obtain 
\begin{align*}
    f(n,n-1) =\sum_{j=0}^{n-1}{\binom{n-j}{j}(-1)^j(n-1-j)!}\,,
\end{align*}
as claimed.

\end{proof}

\begin{rem}\label{rem:conjproved}
    The formula given for the sequence \begin{math}a(n)\end{math} of irreducible permutations on \begin{math}n\end{math} letters, A184185, is:
    \begin{displaymath}
        a(n)=\sum_{k=\lceil \frac{n-1}{2}\rceil}^n{ \binom{k+1}{n-k}(-1)^{n-k}k!}
    \end{displaymath}
    Note that this sum is equivalent to \begin{math}f(n+1,n)\end{math}, which can be seen as follows. Consider the sum expansion for \begin{math}f(n+1,n)\end{math}, given by \cref{thm:f(nn-1)}:
    \begin{displaymath} f(n+1,n)=\sum_{j=0}^{n}{\binom{n+1-j}{j}(-1)^j(n-j)!} \end{displaymath}
    Note that all terms with \begin{math}j>n-j\end{math} are zero since the binomial coefficient evaluates to zero, so we are left with:
    \begin{displaymath}f(n+1,n)=\sum_{j=\lfloor \frac{n}{2}\rfloor+1}^{n}{\binom{n+1-j}{j}(-1)^j(n-j)!}\end{displaymath}
    Setting \begin{math}k=n-j\end{math} and summing over the range of values \begin{math}k\end{math} assumes, we note that since \begin{math}n-\left(\lfloor\frac{n}{2}\rfloor+1\right)=\lceil \frac{n-1}{2}\rceil\end{math}, this variable change yields:
    \begin{align*}
    f(n+1,n) &= \sum_{k=\lceil \frac{n-1}{2}\rceil}^n{ \binom{k+1}{n-k}(-1)^{n-k}k!}\\
    &= a(n)
    \end{align*}
    
    By the bijective construction for \cref{thm:permbij}, \begin{math}f(n+1,n)\end{math} is the number of tree-like tableaux of size \begin{math}n\end{math} containing no occupied corners, therefore these are indeed enumerated by \begin{math}a(n)\end{math}. The conjecture is proven.
\end{rem}

\subsection{Row and Column swaps}

To end this section, we will prove that any CNAT can be transformed into an upper-diagonal one through a series of row and column swaps. Note that this will give an alternative proof of \cref{prop:det}, since the determinant of any upper-diagonal matrix is \begin{math}-1\end{math} and any row or column swap multiplies the determinant by \begin{math}-1\end{math}.

\begin{thm}\label{thm:upperred}
For any \begin{math}n\times n\end{math} CNAT \begin{math}M=M_0\end{math}, we can obtain a sequence of valid CNATs \begin{math}\{M_i\}\end{math} where each \begin{math}M_{i+1}\end{math} was the result of a row or column swap in \begin{math}M_i\end{math}, and for which there is an \begin{math}N\end{math} with \begin{math}M_N\end{math} upper-diagonal. 
\end{thm}

\begin{proof}
We will prove the existence of such row and columns swaps by induction on \begin{math}n\end{math}. Note that a sequence of swaps always exists for \begin{math}n=1, n=2\end{math}.

Firstly, let the leaf in the leftmost column be in row \begin{math}k\end{math}. Note that the only vertex in row \begin{math}k\end{math} is the leaf, and thus we can perform the valid sequence of row swaps
\begin{displaymath}k \leftrightarrow k+1, k+1 \leftrightarrow k+2, \dots, n-1 \leftrightarrow n\end{displaymath}
in order to move the row containing the first column's leaf down to the bottom.

Next, consider the effect of deleting the leftmost column of the CNAT, as well as all edges to vertices in the column. We obtain a forest of smaller CNATs \begin{math}F_i\end{math} (although each is embedded in the larger \begin{math}n\times n\end{math} grid, i.e. Minimality is not satisfied) with any two occupying disjoint rows and columns. We can flatten the labels of each \begin{math}F_i\end{math} (so that Minimality is satisfied), and thus by induction each of the \begin{math}F_i\end{math} can be transformed to upper diagonal form through row and column swaps, so that for any two leaves \begin{math}m_{a,b}, m_{c,d}\end{math} in an \begin{math}F_i\end{math}, if \begin{math}a>c\end{math} then \begin{math}b<d\end{math}. After this, the final upper-diagonal CNM is constructed by "interweaving" column swaps as follows.

Consider the column \begin{math}c \geq 2\end{math} containing the leaf in row \begin{math}n-1\end{math}, and the \begin{math}F_i\end{math} that this leaf belongs to. Since this leaf is the bottommost leaf of \begin{math}F_i\end{math} and \begin{math}F_i\end{math} is upper-diagonal, we must have the root of \begin{math}F_i\end{math} being in the same column \begin{math}c\end{math}. Thus any columns to the left of \begin{math}c\end{math} will not be part of the same \begin{math}F_i\end{math}, therefore we can perform the column swaps \begin{math}c \leftrightarrow c-1, c-1 \leftrightarrow c-2, \cdots,  3 \leftrightarrow 2\end{math} so that the leaf in row \begin{math}n-1\end{math} is correctly in column \begin{math}2\end{math}. All of these column swaps are valid because the \begin{math}F_i\end{math} have disjoint rows and columns, and we did not change any vertices in the first column.

Similarly, consider the column \begin{math}d\end{math} containing the leaf in row \begin{math}n-2\end{math}, and its \begin{math}F_i\end{math}. Note that \begin{math}d\geq 3\end{math} since we already fixed the leaf in column \begin{math}2\end{math}. If this \begin{math}F_i\end{math} has not yet been considered, then the same logic as above applies and we can do the sequence of column swaps \begin{math}d \leftrightarrow d-1, \cdots, 4 \leftrightarrow 3\end{math} to send the leaf in row \begin{math}n-2\end{math} to column \begin{math}3\end{math}. On the other hand if this \begin{math}F_i\end{math} is the same as before, then the same sequence of swaps is still valid because there are no vertices of \begin{math}F_i\end{math} between column \begin{math}2\end{math} (the last column considered) and \begin{math}d\end{math}. This follows from the fact that \begin{math}F_i\end{math} is upper diagonal and the leaves at columns \begin{math}2\end{math} and \begin{math}d\end{math} are on adjacent rows.

This logic can be continued for rows \begin{math}n-3, n-4, \cdots, 2, 1\end{math} at which point the upper-diagonal CNAT will have been constructed.
\end{proof}

\begin{example}
In \cref{fig:rowcolswap} below, we perform the algorithm as described on the 6x6 CNAT in a). Firstly, the leftmost leaf (highlighted in red) is moved to the bottom by swapping rows 4 and 5, then 5 and 6. Then each of the smaller CNATs obtained by deleting the left column are inductively turned into upper-diagonal form - for this example, there are 2 such smaller CNATs highlighted in blue and green, and the blue one is already upper-diagonal. Finally, we interweave the two smaller CNATs together, fixing the leaf in the 5th row with column swaps (in this case by swapping columns 2 and 3), then the 4th row (columns 4 and 5, then 3 and 4), then the 3rd row (already correct), then the 2nd row (columns 4 and 5), then the 1st row is correct.
\end{example}

\begin{figure}[H]
    \centering
    \begin{subfigure}[b]{0.5\textwidth}
        \centering
        \includegraphics{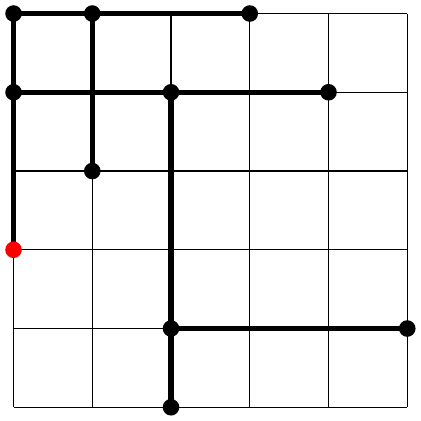}
        \caption{A \begin{math}6 \times 6\end{math} CNAT}
    \end{subfigure}%
    \begin{subfigure}[b]{0.5\textwidth}
        \centering
        \includegraphics{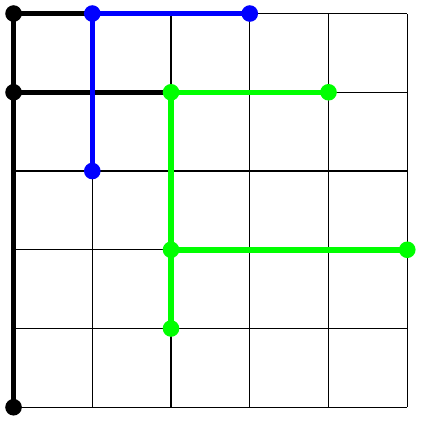}
        \caption{The leftmost leaf moves to the bottom}
    \end{subfigure}
    
    \vspace{10mm}
    \begin{subfigure}[b]{0.5\textwidth}
        \centering
        \includegraphics{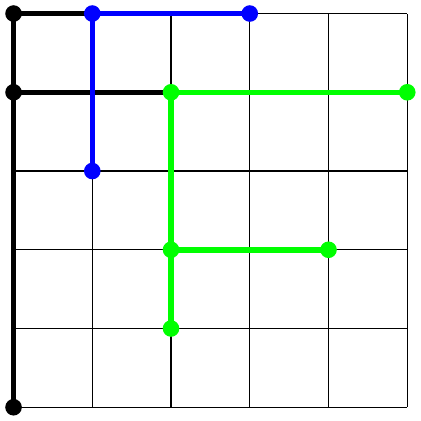}
        \caption{Inductive step}
    \end{subfigure}%
    \begin{subfigure}[b]{0.5\textwidth}
        \centering
        \includegraphics{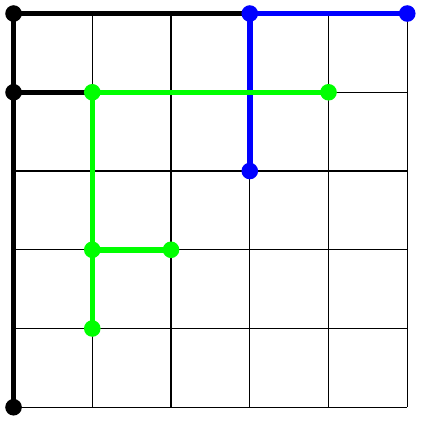}
        \caption{Interweaving}
    \end{subfigure}
    \caption{The row and column swapping process}
    \label{fig:rowcolswap}
\end{figure}

\begin{rem}
We believe that studying these upper-diagonal representatives under an equivalence class may lead to new insight into the underlying structure of CNMs and in particular, provide an opportunity for proving \cref{conj:detparity}.
\end{rem}

\section{The numbers \begin{math}b(n,k)\end{math}}\label{sec:bnk}

\begin{defn}\label{def:bkn}
    Let \begin{math}b(n,k)\end{math} be the number of distinct permutations of \begin{math}n\end{math} letters giving rise to exactly \begin{math}k\end{math} CNATS.
\end{defn}
\begin{thm}[Lower bound on \begin{math}b(n,k)\end{math}]
    Let the numbers \begin{math}b(n,k)\end{math} defined as above. We have the following inequality:
    \begin{displaymath}b(n+1,k)\geq 2\cdot b(n,k)\end{displaymath}
\end{thm}
\begin{proof}
If \begin{math}b(n,k)=0\end{math}, the inequality holds trivially.\\
Thus, assume \begin{math}b(n,k)>0\end{math}. We will prove that every \begin{math}n\times n\end{math} permutation on \begin{math}n\end{math} letters yielding \begin{math}k\end{math} CNATS can be extended in at least two ways to a permutation on \begin{math}n+1\end{math} letters, also resulting in \begin{math}k\end{math} CNATS, proving the result.

Consider the permutation's associated CNAT and in this the leaves on the outermost row or column (See \cref{fig:bound}). By \cref{thm:irred}, exactly two such leaves exist, marked red. Now replace these leaves with zeroes in the CNAT and add ones (blue) on the outermost row and column as in \cref{fig:bound}. Clearly, there is only one way to connect the two new points to form a valid CNAT, this being using the red point.
\begin{figure}[H]
    \begin{subfigure}{0.5\textwidth}
        \centering
        \includegraphics{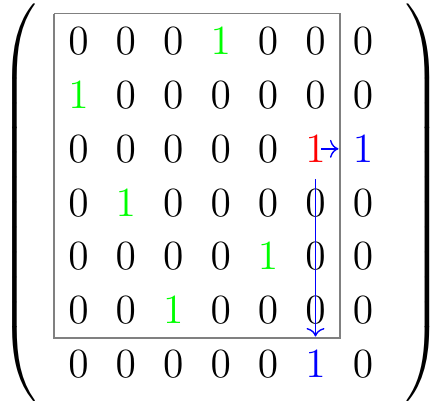}
        \caption{One possible choice}
        \label{fig:bound1}
    \end{subfigure}%
    \begin{subfigure}{0.5\textwidth}
        \centering
        \includegraphics{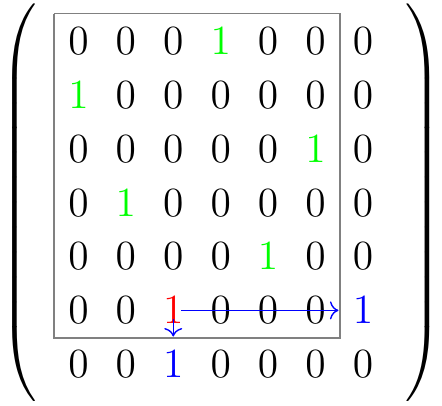}
        \caption{The other choice}
        \label{fig:bound2}
    \end{subfigure}
    \caption{Going from \begin{math}n\end{math} to \begin{math}n+1\end{math}}
    \label{fig:bound}
\end{figure}

Thus every \begin{math}n\times n\end{math} leaf matrix yielding \begin{math}k\end{math} CNATS can be extended into two \begin{math}(n+1)\times (n+1) \end{math} matrices, proving that \begin{math}b(n+1,k)\geq 2\cdot b(n,k)\end{math}.
\end{proof}

Note that \cref{cor:totalunique} established that \begin{math}b(n,0)=2^{n-2}\end{math}, while \cref{cor:totalUDCNM} proved that \begin{math}b(n,(n-1)!)\geq 1\end{math}. Numerical evidence and heuristic argumentation lead us to strongly believe that the following conjecture holds true:

\begin{conj}\label{conj:maxcnats}
    The maximum number of CNATs generated by a permutation \begin{math}\pi\end{math} on \begin{math}n\end{math} letters is \begin{math}(n-1)!\end{math}. This value is only attained when \begin{math}\pi\end{math} is equal to the permutation \begin{math}\left(n, n-1, n-2,\ldots 2,1\right)\end{math}.
\end{conj}

Further numerical experimentation, as seen in \cref{tab:bnk}, led us to conjecture the following:

\begin{conj}\label{conj:b(n2)}
    For all \begin{math}k\geq 1\end{math}, \begin{math}b(n,2)=b(n+1,3)\end{math}.
\end{conj}

\begin{conj}\label{conj:none5}
    For all \begin{math}n\end{math}, there is no permutation \begin{math}\pi\end{math} on \begin{math}n\end{math} letters which gives rise to exactly \begin{math}5\end{math} CNATS, i.e. \begin{math}b(n,5)=0\end{math}.
\end{conj}

\begin{table}[H]
    \centering
    \begin{tabular}{|c|c|c|c|c|c|c|c|c|c|}
        \hline
        \diagbox{\begin{math}n\end{math}}{\begin{math}k\end{math}}  & 1 & 2 & 3 & 4 & 5 & 6 & 7 & 8 & 9   \\
         \hline
        
        2 & 1 & & & & & & & &  \\
      
        3 & 2 & 1 & & & & & & & \\
       
        4 & 4 & 4 & 1 & 3 & 0 & 1 & & &  \\
        
        5 & 8 & 12 & 4 & 14 & 0 & 6 & 2 & 7 & 0\\
        
        6 & 16 & 32 & 12 & 48 & 0 & 24 & 8 & 40 & 1\\
        
        7 & 32 & 80 & 32 & 144 & 0 & 80 & 24 & 160 & 6\\
       
        8 & 64 & 192 & 80 & 400 & 0 & 240 & 64 & 544 & 24\\
        \hline
    \end{tabular}
    \caption{The values of \begin{math}b(n,k)\end{math}}
    \label{tab:bnk}
\end{table}
Naturally, the question arises whether there are any other values \begin{math}k\end{math} other than \begin{math}k=5\end{math} such that \begin{math}b(n,k)=0\end{math} for all \begin{math}n\end{math}. In the very limited range in which we were able to compute, \begin{math}k=5\end{math} was the only such number, however there may well be more.

Sequences \begin{math}b(n,k)\end{math} for \begin{math}n\leq 7\end{math} can be found in \cite{oeis}.
\begin{math}b(n,2), b(n,4), b(n,6), b(n,7)\end{math} appear to be given by the sequences A001787, A176027, A001788, A036289 respectively.

Motivated by a comment on \cite{oeis} for A001787, we computed the binomial transform of the sequences, including one leading zero. 
Surprisingly, \begin{math}b(n,2)\end{math} transformed to 0, 1, 2, 3, 4, 5, 6\begin{math}\dots\end{math}; \begin{math}b(n,4)\end{math} transformed to 0, 3, 8, 15, 24, 35\begin{math}\dots\end{math}, which seems to be \begin{math}(n+1)^2-1\end{math}; \begin{math}b(n,6)\end{math} transformed to the square numbers 0, 1, 4, 9, 16, 25\begin{math}\dots\end{math}; \begin{math}b(n,7)\end{math} transformed to 0, 2, 4, 6, 8\begin{math}\dots\end{math}. We conjecture that these patterns hold for larger \begin{math}n\end{math}.

Interestingly, the pattern of the values of the inverse binomial transform being well-known sequences appears to break down for larger \begin{math}k\end{math}. \begin{math}b(n,8)\end{math}, for example, does not appear in \cite{oeis} and its inverse binomial transform is not an integer polynomial in \begin{math}n\end{math}. 

In a recent pre-print, \cite{sel23} have established \cref{conj:maxcnats}, \cref{conj:b(n2)} and \cref{conj:none5} and proved the binomial transform properties conjectured above for \begin{math}b(n,2)\end{math} and \begin{math}b(n,3)\end{math}. Nevertheless, many questions such as \cref{conj:detparity} remain unproven and further research is needed to fully understand the sequences \begin{math}b(n,k)\end{math}.

\acknowledgments

We are grateful to George Robinson and Sofía Marlasca Aparicio for their insightful commentary and guidance throughout our work. We also appreciate Paul E. Gunnells for his research proposal and helpful remarks. Our thanks go to the anonymous reviewers for their revisions. Finally, we would like to express our sincere gratitude to the PROMYS Europe programme, the University of Oxford and the Clay Mathematics Institute for providing the engaging mathematical environment under which we began our research.

\bibliographystyle{plainnat}
\bibliography{bibliography}
\addcontentsline{toc}{section}{References}

\end{document}